\documentclass{amsart}
\usepackage{amssymb, amsmath, amsfonts, amsthm}

\newcommand{\PP}{\mathbb{P}}

\newcommand{\ZZ}{\mathbb{Z}}

\newcommand{\cF}{\mathcal{F}}

\newcommand{\cU}{\mathcal{U}}

\newtheorem{theorem}{Theorem}

\newtheorem{Lemma}[theorem]{Lemma}

\title{An infinitely generated upper cluster algebra}

\author{David E Speyer}

\begin{document}

\begin{abstract} 
We show that upper cluster algebras need not be finitely generated, answering a question of Berenstein, Fomin and Zelevinsky. Our counter-example is a cluster algebra with $B$-matrix $\left( \begin{smallmatrix} 0 & 3 & -3 \\ -3 & 0 & 3 \\ 3 & -3 & 0 \end{smallmatrix} \right)$ and coefficients obeying a genericity condition. 
\end{abstract}

\maketitle

Let $A$ be an integral domain and let $a_1^+$, $a_1^-$, $a_2^+$, $a_2^-$, $a_3^+$, $a_3^-$ be nonzero elements of $A$ such that, for any positive integer $n$, the products $\prod_i (a_i^+)^n$ and $\prod_i (a_i^-)^n$ are not equal. Let $\cF$ be the field $\mathrm{Frac}(A)(x_1, x_2, x_3)$. Define
\[ \begin{array}{rcl}
x'_1 &=& \frac{a_1^+ x_2^3 + a_1^- x_3^3}{x_1} \\
x'_2 &=& \frac{a_2^+ x_3^3 + a_2^- x_1^3}{x_2} \\
x'_3 &=& \frac{a_3^+ x_1^3 + a_3^- x_2^3}{x_3} \\
\end{array}\]

Let $\cU$ be the intersection of Laurent polynomial rings
\[ A[x_1^{\pm}, x_2^{\pm}, x_3^{\pm}] \cap A[(x'_1)^{\pm}, x_2^{\pm}, x_3^{\pm}] \cap A[x_1^{\pm}, (x'_2)^{\pm}, x_3^{\pm}] \cap A[x_1^{\pm}, x_2^{\pm}, (x'_3)^{\pm}]. \]

The aim of this note is to show that $\cU$ is not finitely generated. 


Let $P$ be a torsion free abelian group, let $A$ be the group algebra $\ZZ P$ and let the $a_i^{\pm}$ be in $P$. Then $\cU$ is an upper cluster algebra with exchange matrix $\left( \begin{smallmatrix} 0 & 3 & -3 \\ -3 & 0 & 3 \\ 3 & -3 & 0 \end{smallmatrix} \right)$, as defined in~\cite{CA3}. More precisely, $\cU$ is the ``upper bound" for the seed $(x_1, x_2, x_3)$ which, by~\cite[Corollary 1.9]{CA3} is equal to the upper cluster algebra. This shows that upper cluster algebras need not be finitely generated; answering~\cite[Problem 1.27]{CA3}.

Let $K = \mathrm{Frac}(A)$. If $\cU$ is finitely generated then so is $\cU \otimes K$, so it is enough to show $\cU \otimes K$ is not finitely generated. $\cU \otimes K$ is the intersection of Laurent polynomial rings
\[ K[x_1^{\pm}, x_2^{\pm}, x_3^{\pm}] \cap K[(x'_1)^{\pm}, x_2^{\pm}, x_3^{\pm}] \cap K[x_1^{\pm}, (x'_2)^{\pm}, x_3^{\pm}] \cap K[x_1^{\pm}, x_2^{\pm}, (x'_3)^{\pm}]. \]
So we can immediately reduce to the case $A =K$, that is to say, that $A$ is a field.

The notation above was chosen to make the connection to cluster algebras as clear as possible. We now rename our terms for readability.
We write $(x,y,z)$ rather than $(x_1, x_2, x_3)$, and $x'$, $y'$, $z'$ in place of $x'_1$ etcetera. 
Define the polynomials $p = a_1^{+} y^3 + a_1^{-1} z^3$, $q = a_2^+ z^3 + a_2^- x^3$ and $r = a_3^+ x^3 + a_3^- y^3$.
Note that $\cU$ is a graded algebra, where $\deg(x)=\deg(y) = \deg(z) = 1$ and $K$ is in degree $0$. 
We write $\cU_e$ for the degree $e$ part of $\cU$.


\begin{Lemma} \label{vanishing}
Let $f(x,y,z) = \sum_{a+b+c=d} f_{abc} x^a y^b z^c$ be a homogenous polynomial in $K[x,y,z]$. Then $\frac{f(x,y,z)}{x^i y^j z^k}$ is in $\cU$ if and only if 
\begin{itemize}
\item $p^{i-a}$ divides $\sum_{b+c=d-a} f_{abc} y^b z^c$ for all $a < i$
\item $q^{j-k}$ divides $\sum_{a+c=d-b} f_{abc} x^b z^c$ for all $b < j$
\item $r^{k-c}$ divides $\sum_{a+b=c-d} f_{abc} x^a y^b$ for all $c < k$. 
\end{itemize}
Here the sums are over the variables on the left hand side of each summation condition.
\end{Lemma}

\begin{proof}
We will show that $\frac{f(x,y,z)}{x^i y^j z^k}$ is in $K[x'^{\pm}, y^{\pm}, z^{\pm}]$ if and only if $p^{i-a}$ divides $\sum_{b+c=d-a} f_{abc} y^b z^c$. Similar arguments give similar criteria for when $\frac{f(x,y,z)}{x^i y^j z^k}$ is in the other Laurent polynomial rings defining $\cU$.

Write $f(x,y,z) = \sum_a x^a f_a(y,z)$. Then  
\[ \frac{f(x,y,z)}{x^i y^j z^k} = \sum_a \frac{p(y,z)^{a-i} f_a(y,z)}{(x')^{a-i} y^j z^k} .\]
 The summands all have different degrees in the grading where $\deg x'=1$ and $\deg y = \deg z =0$. So the only ways that this sum can be a Laurent polynomial is if each summand is a Laurent polynomial. The terms with $a \geq i$ are obviously Laurent polynomials; the condition that $ p(y,z)^{a-i} f_a(y,z)/((x')^{a-i} y^j z^k)$ is a Laurent polynomial is precisely that $p^{i-a}$ divides $f_a$.
\end{proof}

\begin{Lemma} \label{deg0}
The degree $0$ part of $\cU$ is the field $K$.
\end{Lemma}

\begin{proof}
Suppose that $\frac{f(x,y,z)}{x^i y^j z^k}$  is in $\cU$ with $\deg f = i+j+k$ and $f \neq x^i y^j z^k$. 
We may assume that this fraction is in lowest terms, so that $x$, $y$ and $z$ don't divide $f$. Without loss of generality, assume that $i \geq j \geq k$, so $3i \geq i+j+k$.

Then $f(0,y,z)$ is divisible by $p^i$. Since $\deg f = i+j+k \leq 3i$ and $\deg p^i = 3i$, this shows that $f(0,y,z)$ is a scalar multiple of $p^i$ and, as $x$ does not divide $f$, it is a nonzero scalar multiple. Furthermore, we must have $i+j+k=3i$, and hence $i=j=k$. We therefore likewise have that $f(x,0,z)$ is a scalar multiple of $q^i$ and $f(x,y,0)$ is a scalar multiple of $r^i$. Let $A$, $B$ and $C$ be the coefficients of $x^i$, $y^i$ and $z^i$ in $f$. Then we have
\[ \left( \frac{a_1^{+} a_2^{+} a_3^{+}}{a_1^{-} a_2^{-} a_3^{-}} \right)^i = \frac{BCA}{CAB} = 1\]
contradicting our choice that $(a_1^{+} a_2^{+} a_3^{+})/(a_1^{-} a_2^{-} a_3^{-})$ is not torsion in $K^{\times}$.
\end{proof}


\begin{Lemma} \label{deg1}
For every positive integer $d$, the vector space of polynomials of degree $3d+1$ which vanish to order $d$ at the $p$'s, $q$'s and $r$'s has dimension at least $3d+3$.
\end{Lemma}

I believe the correct number is always exactly $3d+3$, but I don't need to know for what follows. 

\begin{proof}
Let $V$ be the vector space of polynomials of degree $3d+1$ in $x$, $y$ and $z$, so $\dim V = (3d+3)(3d+2)/2$. 
The condition that a polynomial be divisible by $p^{i-a}$ imposes $3(i-a)$ linear conditions. So, writing an element $f$ in $V$ as $\sum x^a f_a(y,z)$, the condition that $f_a$ be divisible by $p^{d-a}$ imposes $3 \left( 1+2+\cdots + d \right) = 3d (d+1)/2$ linear conditions on $V$.
So the space of polynomials in $V$ which obey the divisibility conditions with respect to $p$, $q$ and $r$ has dimension at least
\[ \dim V - 9 d(d+1)/2 = 3d+3 . \qedhere \] 
\end{proof}
%

We now prove that $\cU$ is not finitely generated. Suppose otherwise, for the sake of contradiction.
$\cU$ is a graded ring, and $\cU_0$ is simply the field $K$ by Lemma~\ref{deg0}. So, if $\cU$ is finitely generated, then $\cU_1$ is a finite dimensional $K$ vector space.
But, for every integer $d$, Lemma~\ref{deg1} shows that $\cU_1$ has dimension $\geq 3d+3$, a contradiction.

\textbf{Remark:} From a geometric perspective, let $K$ be algebraically closed of characteristic $\neq 3$. Let $\PP^2$ be the projective plane over $K$. The polynomial $p$ vanishes at $3$ points on the line $(0 : \ast : \ast)$; call them $p_1$, $p_2$, $p_3$. Similarly, let $q_1$, $q_2$, $q_3$ denote the $3$ zeroes of $q$ on $(\ast : 0 : \ast)$ and let $r_1$, $r_2$, $r_3$ denote the $3$ zeroes of $r$ on $(0 : \ast : \ast)$. Let $X$ be the surface formed by blowing up these $9$ points of $\PP^2$, and the deleting the proper transforms of the coordinate lines. Lemma~\ref{vanishing} says that $\cU$ is the algebra of global functions on $X$. Nagata's famous example of a non-finitely generated of rings of invariants also involved blowing up $9$ points in $\PP^2$ and deleting the proper transform of a cubic; see \cite[Chapter 4]{Dolg} for a presentation of Nagata's result similar to our argument here.
Gross, Hacking and Keel (unpublished) have begun studying the moduli of surfaces with an anti-canonical from the perspective of cluster algebras, and this example should be a special case of their theory.

\end{document}